\documentclass[12pt]{amsart}

\usepackage{amssymb,amsmath,latexsym,amsthm}
\usepackage[english]{babel}
\usepackage[all]{xy}
\usepackage{graphicx}

\newtheorem{lemma}{Lemma}[section]
\newtheorem{proposition}{Proposition}[section]
\newtheorem{cor}{Corollary}[section]

\theoremstyle{definition}

\newtheorem{example}{Example}

\begin{document}

\title{On some branches of the Bruhat-Tits tree}


\author{{\sc Luis Arenas-Carmona} and\\
{\sc Ignacio Saavedra}}


\newcommand\Q{\mathbb Q}
\newcommand{\h}{\mathfrak{H}}
\newcommand\alge{\mathfrak{A}}
\newcommand\Da{\mathfrak{D}}
\newcommand\Ea{\mathfrak{E}}
\newcommand\Ha{\mathfrak{H}}
\newcommand\oink{\mathcal O}
\newcommand\matrici{\mathbb{M}}
\newcommand\Txi{\lceil}
\newcommand\ad{\mathbb{A}}
\newcommand\enteri{\mathbb Z}
\newcommand\finitum{\mathbb{F}}
\newcommand\bbmatrix[4]{\left(\begin{array}{cc}#1&#2\\#3&#4\end{array}\right)}
\newcommand\lbmatrix[4]{\textnormal{\scriptsize{$\left(\begin{array}{cc}#1&#2\\#3&#4
\end{array}\right)$}}\normalsize}
\newcommand\lcvector[2]{\textnormal{\scriptsize{$\left(\begin{array}{c}#1\\#2\end{array}\right)$}}\normalsize}

\maketitle

\begin{abstract}
We give an algorithm to explicitly compute the largest subtree, in the local Bruhat-Tits tree for $\mathrm{PSL}_2(k)$,
whose vertices correspond to orders containing a given suborder $\Ha$, in terms of a set of generators for $\Ha$.
The shape of this subtree is described, when it is finite, by a set of two invariants. We use our method to provide a full
table for the invariants of an order generated by a pair of orthogonal pure quaternions. 
In a previous work, the first author showed that determining the shape of these local subtrees allows the computation of representation fields, a class field determining the set of isomorphism classes, in a genus $\mathbb{O}$ of orders of maximal rank in a fixed central simple algebra, containing an isomorphic copy of $\Ha$. Some further applications are described here.
\end{abstract}

\bigskip
\section{Introduction}

Let $k$ be a local  field, and let $\oink$ be the ring of integers of $k$.
Recall that the set of maximal orders in $\matrici_2(k)$ is in correspondence with the set of vertices in the Bruhat-Tits tree (or BT-tree) for $PSL_2(k)$ \cite{trees}, \cite{vigneras}. In \cite{eichler2} we proved that the subtree $S_0(\Ha)$ whose vertices are the maximal orders containing a given suborder lies in a rather restricted family. In fact, for most orders, the tree $S_0(\Ha)$, also called the branch of $\Ha$ in what follows,  is the maximal subtree whose orders lie no farther than a fixed distance, the depth $p=p(\Ha)$, from a path (Figure 1b), called the stem of $\Ha$, which can have length 0 (Figure 1a), be infinite in one (Figure 1c), or both ends (Figure 1d). Such a set is called a $p$-thick line, or a thick line if the depth is irrelevant. The exceptions are the following:
\begin{itemize}
\item If $\Ha$ is the rank 2 order generated by a nilpotent element, $S_0(\Ha)$ is an infinite leaf, which can be thought as a thick line with the stem at infinity (Figure 1e),
\item if $\Ha=\oink$, identified with the ring of integral scalar matrices, then $S_0(\Ha)$ is the whole BT-tree.
\end{itemize} 
In particular, for any order contained in finitely many maximal orders, as is the cases for orders of maximal rank, the branch $S_0(\Ha)$ is completely described by two invariants, the depth $p$, and the stem length $l=l(\Ha)$. In fact, the intersection of the maximal orders containing $\Ha$ is $$\Da^{[p]}=\oink\mathbb{I}+\pi^p\Da,$$ where $\Da$ is an Eichler order of level $l$, and $\pi$ is any uniformizing parameter of $k$. These intersections were originally described by Tu \cite{tu}.

\begin{figure}
\unitlength 1mm 
\linethickness{0.4pt}
\ifx\plotpoint\undefined\newsavebox{\plotpoint}\fi 
\begin{picture}(120,16)(0,0)
\put(10,14){\makebox(0,0)[cc]{a}}
\put(10,6){\line(0,1){4}}
\multiput(10,6)(-0.087,-.05){40}{\line(1,0){.0625}}
\multiput(10,6)(0.087,-.05){40}{\line(1,0){.0625}}
\multiput(10,10)(-0.087,.05){20}{\line(1,0){.0625}}
\multiput(10,10)(0.087,.05){20}{\line(1,0){.0625}}
\put(6.6,4){\line(0,-1){2}}
\multiput(6.6,4)(-0.087,.05){20}{\line(1,0){.0625}}
\put(13.4,4){\line(0,-1){2}}
\multiput(13.4,4)(0.087,.05){20}{\line(1,0){.0625}}
\put(31,14){\makebox(0,0)[cc]{b}}
\put(25,6){\line(1,0){4}}
\put(29,6){\line(0,1){4}}
\multiput(29,10)(-0.087,0.05){20}{\line(1,0){.0625}}
\multiput(29,10)(0.087,0.05){20}{\line(1,0){.0625}}
\multiput(25,6)(-0.05,-0.087){40}{\line(1,0){.0625}}
\multiput(25,6)(-0.05,0.087){40}{\line(1,0){.0625}}
\put(23,9.28){\line(-1,0){2}}
\put(23,2.72){\line(-1,0){2}}
\multiput(23,9.28)(0.05,0.087){20}{\line(1,0){.0625}}
\multiput(23,2.72)(0.05,-0.087){20}{\line(1,0){.0625}}
\put(29,6){\line(1,0){5}}
\put(34,6){\line(0,1){4}}
\multiput(34,10)(-0.087,0.05){20}{\line(1,0){.0625}}
\multiput(34,10)(0.087,0.05){20}{\line(1,0){.0625}}
\put(34,6){\line(1,0){4}}
\multiput(38,6)(0.05,-0.087){40}{\line(1,0){.0625}}
\multiput(38,6)(0.05,0.087){40}{\line(1,0){.0625}}
\put(40,9.28){\line(1,0){2}}
\put(40,2.72){\line(1,0){2}}
\multiput(40,9.28)(-0.05,0.087){20}{\line(1,0){.0625}}
\multiput(40,2.72)(-0.05,-0.087){20}{\line(1,0){.0625}}
\put(54,14){\makebox(0,0)[cc]{c}}
\put(50,6){\line(1,0){4}}
\put(54,6){\line(0,1){4}}
\multiput(54,10)(-0.087,0.05){20}{\line(1,0){.0625}}
\multiput(54,10)(0.087,0.05){20}{\line(1,0){.0625}}
\multiput(50,6)(-0.05,-0.087){40}{\line(1,0){.0625}}
\multiput(50,6)(-0.05,0.087){40}{\line(1,0){.0625}}
\put(48,9.28){\line(-1,0){2}}
\put(48,2.72){\line(-1,0){2}}
\multiput(48,9.28)(0.05,0.087){20}{\line(1,0){.0625}}
\multiput(48,2.72)(0.05,-0.087){20}{\line(1,0){.0625}}
\put(54,6){\line(1,0){5}}
\put(59,6){\line(0,1){4}}
\multiput(59,10)(-0.087,0.05){20}{\line(1,0){.0625}}
\multiput(59,10)(0.087,0.05){20}{\line(1,0){.0625}}
\multiput(60,6)(1,0){4}{{\rule{.4pt}{.4pt}}}
\put(76,14){\makebox(0,0)[cc]{d}}
\put(71,6){\line(0,1){4}}
\multiput(71,10)(-0.087,0.05){20}{\line(1,0){.0625}}
\multiput(71,10)(0.087,0.05){20}{\line(1,0){.0625}}
\multiput(71,6)(-1,0){4}{{\rule{.4pt}{.4pt}}}
\put(71,6){\line(1,0){5}}
\put(76,6){\line(0,1){4}}
\multiput(76,10)(-0.087,0.05){20}{\line(1,0){.0625}}
\multiput(76,10)(0.087,0.05){20}{\line(1,0){.0625}}
\put(76,6){\line(1,0){5}}
\put(81,6){\line(0,1){4}}
\multiput(81,10)(-0.087,0.05){20}{\line(1,0){.0625}}
\multiput(81,10)(0.087,0.05){20}{\line(1,0){.0625}}
\multiput(82,6)(1,0){4}{{\rule{.4pt}{.4pt}}}
\put(102,14){\makebox(0,0)[cc]{e}}
\put(94,6){\line(0,1){4}}
\multiput(94,10)(-0.087,0.05){20}{\line(1,0){.0625}}
\multiput(95.8,11)(0.087,-0.05){10}{\line(1,0){.0625}}
\put(95.8,11){\line(0,1){1}}
\multiput(94,10)(0.087,0.05){20}{\line(1,0){.0625}}
\multiput(92.2,11)(-0.087,-0.05){10}{\line(1,0){.0625}}
\put(92.2,11){\line(0,1){1}}
\multiput(94,6)(-1,0){4}{{\rule{.4pt}{.4pt}}}
\put(94,6){\line(1,0){8}}
\put(102,6){\line(0,1){4}}
\multiput(102,10)(-0.087,0.05){20}{\line(1,0){.0625}}
\multiput(102,10)(0.087,0.05){20}{\line(1,0){.0625}}
\put(102,6){\line(1,0){5}}
\put(107,6){\line(0,1){4}}
\put(107,6){\line(1,0){4}}
\end{picture}
\caption{Some examples of branches of orders when the residue field is
$\mathbb{K}=\oink/\pi\oink\cong\finitum_2$.} 
\end{figure}
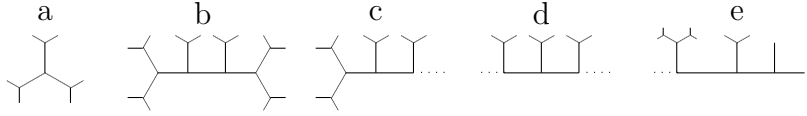

In this work we show how these invariant can be computed in terms of a generating set for $\Ha$. As an example we compute all branches of orders of the form $\Ha=\oink[i,j]$ where $i$ and $j$ are a pair of orthogonal pure quaternions, the standard generators of a quaternion algebra. These computations are useful to study which Eichler orders represent a given suborder. A fundamental tool for this study is the representation field, whose definition is recalled in the last section. The relative spinor image, which permits the computation of representation fields, is easily expressed in terms of a third invariant, the diameter $d(\Ha)=l(\Ha)+2p(\Ha)$. It is the diameter of $S_0(\Ha)$ as a metric space with the canonical metric  in a tree. Other applications of the invariants are presented in \S5. For instance, we provide a formula for the number of maximal orders containing a given order of maximal rank, or any local order with finite invariants.

If $S$ is a thick line, we also denote its invariants by $d(S)$, $l(S)$ and $p(S)$. When the stem of $S$ is a maximal path, we write $d(S)=l(S)=\infty$. When the stem of $S$ is infinite in one direction only, we write $d(S)=l(S)=\infty/2$.
In all that follows, when $\Ha=\oink[a_1,\dots,a_n]$ we write $S_0(a_1,\dots,a_n)$ instead of $S_0(\Ha)$. Note that
$S_0(a_1,\dots,a_n)=\bigcap_{i=1}^nS_0(a_i)$.
This property is used in \cite{eichler2} to characterize branches of orders among all possible subtrees of the BT-tree. The intersection of every non-empty sub-collection of the sets on the right of this identity is a thick line or an infinity leaf,
so it suffices to describe the intersection of any two of these sets in terms of their relative positions. Then we need to determine these relative positions for explicit orders.  
This needs to be done by an ad-hoc argument in every specific case, but Lemma \ref{l31} is particularly useful there. The possible shapes for each branch $S=S_0(a_i)$ 
are fully described by the following list \cite[Cor. 4.3 and Prop. 4.4]{eichler2}:
\begin{enumerate}
\item If $L=k(a_i)$ is a field, then $S$ is a thick line with a stem of length 1 if $L/k$ is ramified and 0 otherwise.
\item If $L=k(a_i)\cong k\times k$, then $S$ is a thick line with $l(S)=\infty$.
\item If $k(a_i)$ contains a non-trivial nilpotent element, then $S$ is an infinite leaf (Fig. 1e). \end{enumerate}
Furthermore, in either of the first two cases, the depth $p=p(\oink[a_i]\big)$ is defined by the identity $\oink[a_i]=\oink_L^{[p]}=\oink+\pi^p\oink_L$, where $\oink_L$ is the maximal order of $L$.

\section{Explicit intersections}

Let $G$ be a graph and let $\delta$ be the usual distance on it. For any subgraph $S$ of $G$ we define
 $S^{[r]}=\bigcup_{x\in S}B[x;r]$, where $B[x;r]=B_\delta[x;r]$ is the closed ball of radius $r$ centered at $x$. 
The length of a path $\gamma$ is denoted $l(\gamma)$.
\begin{lemma}\label{lfund}
If $G$ is a tree, while $S_1$ and $S_2$ are two subtrees with non-empty intersection, we have
$(S_1\cap S_2)^{[r]}=S_1^{[r]}\cap S_2^{[r]}$.
\end{lemma}

\begin{proof}
It is clear that $(S_1\cap S_2)^{[r]}\subseteq S_1^{[r]}\cap S_2^{[r]}$, so we prove the opposite inclusion.
First observe that since the path joining two vertices in a tree is unique, if $S_1$ and $S_2$ are path-connected so is their intersection.  Now assume $x\in  S_1^{[r]}\cap S_2^{[r]}$. Let $\alpha_i$ be the shortest path joining $x$ to a point
$y_i\in S_i$ (Fig. 2A), so that $l(\alpha_i)=\delta(x,y_i)\leq r$. Let $\gamma_i$ be the shortest path (in $S_i$) joining $y_i$ to a point $z_i\in S_1\cap S_2$. Let $\beta$ be the path (in $S_1\cap S_2$)
joining $z_1$ and $z_2$. By the uniqueness of paths either of the following statements hold:
\begin{itemize}\item $\alpha_2$ passes through $y_1$, $z_1$, and $z_2$
(Fig. 2B), and therefore $l(\beta)=l(\gamma_2)=0$ and $z_1=z_2=y_2$ by definition of $\alpha_2$, or
 \item $\alpha_1$ passes through $y_2$, $z_2$, and $z_1$ (Fig. 2C),
and therefore $l(\beta)=l(\gamma_1)=0$ and $y_1=z_1=z_2$. 
\end{itemize}
In the first case, the distance from $x$ to  $z_1=y_2$ is $l(\alpha_2)\leq r$. The remaining case is analogous.
\end{proof}
\begin{figure}
\[ 
\includegraphics[width=1\textwidth]{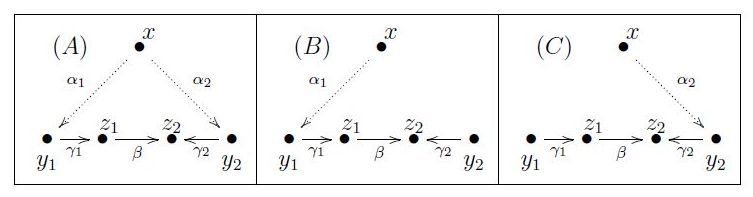}
\]
\caption{Paths in the proof of Lemma \ref{lfund}. In (B) $\alpha_2$ equals the juxtaposition $\alpha_1*\gamma_1*\beta*\gamma_2^{-1}$. In (C) we have 
$\alpha_1=\alpha_2*\gamma_2*\beta^{-1}*\gamma_1^{-1}$.} 
\end{figure}

\begin{cor}
In the notations of the previous lemma, if $S_1\cap S_2$ is a $p$-thick line, then for every
positive integer $t$, the intersection $S_1^{[t]}\cap S_2^{[t]}$
is a $(p+t)$-thick line with the same stem. 
\end{cor}

\begin{lemma}\label{lfund2}
If $G$ is a tree, while $S_1$ and $S_2$ are two subtrees satisfying $S_1\cap S_2=\emptyset$ and
$S_1^{[1]}\cap S_2^{[1]}\neq\emptyset$, then $S_1^{[1]}\cap S_2^{[1]}$ is either a point or a path
of length $1$. 
\end{lemma}

\begin{proof}
Let $x\in S_1^{[1]}\cap S_2^{[1]}$. Then $x\notin S_1\cap S_2$ as the latter set is empty, whence $x\notin S_1$ or $x\notin S_2$.
If  $x\notin S_1$, then $x$ is an endpoint of  $S_1^{[1]}$. We conclude that $x$ has a unique neighbor $y$ in $S_1$.
In particular, $y$ is not in $S_2$, whence $y$ is either an endpoint of  $S_2^{[1]}$, or $y\notin  S_2^{[1]}$.
In the first case, $S_1^{[1]}\cap S_2^{[1]}$ is the path joining $x$ and $y$, while in the latter case
$S_1^{[1]}\cap S_2^{[1]}=\{x\}$.
\end{proof}

In all that follows we write $[\alpha]$ and $\{\alpha\}=\alpha-[\alpha]$ for the integral part and the fractional part, respectively, of a real number $\alpha$. We also use the conventions $\mathrm{min}(\alpha,\infty/2)=\alpha$,
$\alpha+\infty/2=\infty/2$, and $\infty/2+\infty/2=\infty$.

\begin{proposition}\label{inter1}
Let $S_1$ and $S_2$ be two thick lines whose stems $T_1$ and $T_2$ lie at a distance $e>0$ from each other, and
let $a$, $b$, $c$, and $d$, be the length of the segments of the stems determined by the unique path joining them as in Fig. 3A. Let $l_1=a+b$ and $l_2=c+d$ be the stem length of $S_1$ and $S_2$, and let $p_1$ and $p_2$ be the corresponding depths. Then:
\begin{enumerate}
\item if $e>|p_1-p_2|$, then $S_3=S_1\cap S_2$ is a thick line with invariants $p_3=\left[\frac{p_1+p_2-e}2\right]$
and $l_3=2\left\{\frac{p_1+p_2-e}2\right\}\in\{0,1\}$,
\item if $0<e\leq p_1-p_2$, then $S_3=S_1\cap S_2$ is a thick line with invariants $p_3=p_2$
and $l_3=\mathrm{min}\{\tau,c\}+\mathrm{min}\{\tau,d\}$, where $\tau=p_1-p_2-e$,
\item if $0<e\leq p_2-p_1$, then $S_3=S_1\cap S_2$ is a thick line with invariants $p_3=p_1$
and $l_3=\mathrm{min}\{\mu,a\}+\mathrm{min}\{\mu,b\}$, where $\mu=p_2-p_1-e$.
'\end{enumerate}
\end{proposition}
\begin{figure}
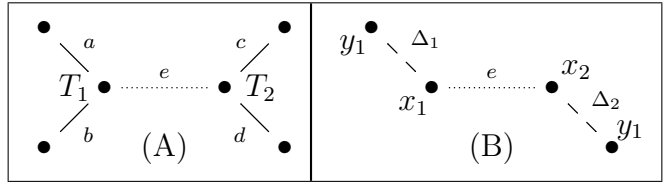

\[ \fbox{ \xygraph{
!{<0cm,0cm>;<.8cm,0cm>:<0cm,.8cm>::} !{(1,0) }*+{\bullet}="a" !{(0.5,0) }*+{T_1}="a1"
!{(3,0)}*+{\bullet}="b"  !{(3.6,0) }*+{T_2}="b1" !{(0,1) }*+{\bullet}="c"
!{(0,-1) }*+{\bullet}="d" !{(4,1)}*+{\bullet}="e"
!{(4,-1)}*+{\bullet}="f"
 !{(2,-1) }*+{\mathrm{(A)}}="z"
 "a"-@{.}^e"b" "a"-_a"c" "a"-^b"d" "b"-^c"e"
"b"-_d"f" } }
\fbox{ \xygraph{
!{<0cm,0cm>;<.8cm,0cm>:<0cm,.8cm>::} !{(1,0) }*+{\bullet}="a" !{(0.7,-0.3) }*+{x_1}="a1"
!{(3,0)}*+{\bullet}="b" !{(-0.3,0.7) }*+{y_1}="b1" !{(3.4,0.3) }*+{x_2}="b1" !{(0,1) }*+{\bullet}="c"
!{(4,-1)}*+{\bullet}="e"  !{(4.3,-0.7) }*+{y_1}="e1"  !{(2,-1) }*+{\mathrm{(B)}}="z"
 "a"-@{.}^e"b" "a"-@{--}_{\Delta_1}"c"  "b"-@{--}^{\Delta_2}"e"} }
\]
\caption{Disposition of the stems in Proposition \ref{inter1} (A), and  the vertices used in the proof (B).} 
\end{figure}

\begin{proof}
Assume first that $e>|p_1-p_2|$. Without loss of generality we can assume $p_2\leq p_1$. Now, we can write $S_i=U_i^{[p_2]}$,
for $i\in\{1,2\}$, where $U_2$ is a path, while $U_1$ is a $(p_1-p_2)$-thick line not intersecting $U_2$. We conclude that, for some positive integer $t\leq p_2$, we have  $U_1^{[t-1]}\cap U_2^{[t-1]}=\emptyset$ and $U_1^{[t]}\cap U_2^{[t]}\neq\emptyset$. We conclude from Lemma \ref{lfund2} and the corollary to Lemma \ref{lfund} that  $l(S_3)\in\{0,1\}$. Since the
invariants satisfy the relation $d(S_3)=l(S_3)+2p(S_3)$, it suffices to compute the diameter. We claim that the
diameter of $S_3$ is $d=p_1+p_2-e$. The result follows easily from the claim. Now let $x_1$ be the point of $T_1$ that is closest to $T_2$, and let $x_2$ be the point of $T_2$ that is closest to $T_1$. Let $y_1$ be a point in the BT-tree 
such that the path from $y_1$ to $x_2$  passes through $x_1$ and set $\delta(y_1,x_1)=\Delta_1$.
Define $y_2$ and $\Delta_2$ analogously as in
Figure 3B.  Then $y_1$ belongs to $S_2$ if and only if $\Delta_1\leq p_2-e$, and all points satisfying this condition are in $S_1$. Similarly, $y_2$ belongs to $S_1$ if and only if $\Delta_2\leq p_1-e$. The result follows.

Assume now that $0<e\leq p_1-p_2$. Then we can write $S_i=U_i^{[p_2]}$,
for $i\in\{1,2\}$, where $U_2=T_2$ is a path, while $U_1$ is a $(p_1-p_2)$-thick line intersecting $U_2$.
In this case $U_1\cap U_2$ is a path of length $l_3$ as in the statement, so the result follows. 
 The final case is analogous.
\end{proof}

\begin{proposition}\label{inter2}
Let $S_1$ and $S_2$ be two thick lines whose stems $T_1$ and $T_2$  intersect as shown in Figure 4, 
so that $l_1=a+e+b$ and $l_2=c+e+d$ are their respective stem lengths.
 Let $p_1$ and $p_2$ be the depths of $S_1$ and $S_2$, respectively. Then $S_3=S_1\cap S_2$ has invariants
$p_3=\mathrm{min}\{p_1,p_2\}$ and 
$$l_3=\left\{\begin{array}{rcl}
e+\mathrm{min}\{a,p_2-p_1\}+\mathrm{min}\{b,p_2-p_1\}&\textnormal{ if }&p_1\leq p_2,\\
e+\mathrm{min}\{c,p_1-p_2\}+\mathrm{min}\{d,p_1-p_2\}&\textnormal{ if }&p_2\leq p_1\end{array}\right..
$$
\end{proposition}
\begin{figure}
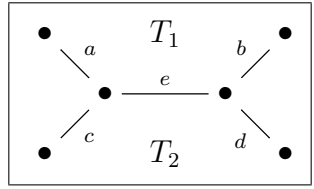

\[ \fbox{ \xygraph{
!{<0cm,0cm>;<.8cm,0cm>:<0cm,.8cm>::} !{(1,0) }*+{\bullet}="a" !{(2,1) }*+{T_1}="a1"
!{(3,0)}*+{\bullet}="b"  !{(2,-1) }*+{T_2}="b1" !{(0,1) }*+{\bullet}="c"
!{(0,-1) }*+{\bullet}="d" !{(4,1)}*+{\bullet}="e"
!{(4,-1)}*+{\bullet}="f" "a"-^e"b" "a"-_a"c" "a"-^c"d" "b"-^b"e"
"b"-_d"f" } }
\]
\caption{Disposition of the stems in Proposition \ref{inter2}.} 
\end{figure}

\begin{proof}
Assume $p_1\leq p_2$. Reasoning as before, we can assume that $p_1=0$, so $S_1=T_1$ is a line. The result follows.
The other case is analogous.
\end{proof}

In all that follows, for any metric space $X$ and any subset $A$, we define a relative depth function $p(A,-):A\rightarrow\mathbb{R}$ by
$$p(A,x)=\sup\left\{r\in[0,\infty)\Big|B(x;r)\subseteq A\right\}.$$
 It is apparent that for any two sets $A_1,A_2\subseteq X$
we have $p(A_1\cap A_2,x)=\mathrm{min}\{p(A_1,x),p(A_2,x)\}$, for any point $x\in A_1\cap A_2$. This concept is useful to us since for any order $\Ha$
with a finite branch, we have $$p(\Ha)=\max_{x\in S_0(\Ha)}p\big(S_0(\Ha),x\big),$$ while this maximum is attained precisely at the points in the stem. 

Recall from \cite{eichler2} that the points in the branch of an order can be classified into stem and leaf points according to the depth of each vertex and its neighbors, as shown in Fig. 5. Note that, from a leaf point there is always a unique direction leading to vertices of larger depth, whence every path without backtracking  in a different direction is a path going outwards through the leaves, a path that cannot be longer that the depth of the starting point. In an infinite leaf $S$, which has no stem points, every path without backtracking has at most one vertex at maximal depth, and  the path on every side from such vertex goes outwards through the leaves. Such paths can only intersect the infinite leaf on a finite path. On the other hand, starting from every vertex of the infinite leaf, there is a unique path without backtracking leading always in the direction of higher depth. Such paths, which are infinite on one direction, are called long paths of the leaf. The long paths corresponding to two vertices $\Da$ and $\Da'$ always coincide from some point $\Da''$ onwards, and the segments determined by $\Da''$ on each path, as in Fig. 6A,
satisfy $p(S,\Da'')=a+p(S,\Da)=b+p(S,\Da')$, so in particular, if $\Da$ and $\Da'$ are endpoints of $S$, i.e. their relative depths are zero, we have $a=b$.

\begin{figure}
 \scriptsize
\[ 
\includegraphics[width=1\textwidth]{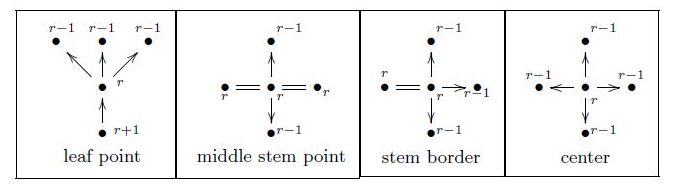}
\]
\normalsize 
\caption{Stem or leaf points. Arrows indicate "outwards" directions. Double lines denote stem edges.}
\end{figure}
\begin{figure}
\[ \includegraphics[width=1\textwidth]{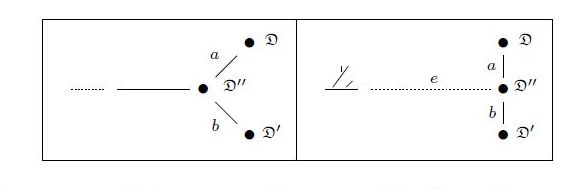}
\]
\caption{Two long paths in an infinite leaf  (A), and a stem not intersecting an infinite leaf (B).}
\end{figure}

\begin{proposition}\label{inter3}
Let $S_1$ be a thick line and let $S_2$ be an infinite leaf. Let $T_1$ be the stem of $S_1$. 
Let $p_3$ and $l_3$ be the invariants of $S_3=S_1\cap S_2$,
 while $p_1$ is the depth of $S_1$. Then,
exactly one of the following conditions hold:
\begin{enumerate}
\item $T_1$ has a vertex $\Da''$ at maximal depth $p_0>p_1$ in $S_2$, which divides $T_1$ in two parts of lengths
$a$ and $b$. In this case $p_3=p_1$ and $l_3=\mathrm{min}\{a,p_0-p_1\}+\mathrm{min}\{b,p_0-p_1\}$.
\item $T_1$ has a vertex $\Da''$ at maximal depth $p_0$, with $0\leq p_0\leq p_1$. In this case $p_3=[\frac{p_1+p_0}2]$, and
$l_3=2\{\frac{p_1+p_0}2\}\in\{0,1\}$.
\item $T_1$ contains a long path of $S_2$. In this case $l_3=\infty/2$, and $p_3=p_1$.
\item $T_1$ lies at a positive distance $e\leq p_1$ from the infinite leaf $S_2$ (Fig. 6B).
 In this case,  $p_3=[\frac{p_1-e}2]$, and
$l_3=2\{\frac{p_1-e}2\}\in\{0,1\}$.
'\end{enumerate}
\end{proposition}

\begin{proof}
Note that, if $e=\delta(T_1,S_2)>p_1$, then $S_1\cap S_2=\emptyset$.
It follows from the paragraph preceding the proposition that (1)-(4) are actually all possible cases.
Assume the hypotheses in case 1. If $p_1=0$, so that $S_1=T_1$ is a path, then so is $S_3$, 
and the result follows from a quick look at Figure 6A. More generally, note that $S_2=U_t^{[t]}$ 
for an  arbitrary positive integer $t$ and a suitable infinite leaf $U_t$. Since $p(U_{p_1},\Da'')=p(S_2,\Da'')-p_1$,
the result follows in this case.

Assume the hypotheses in case 3. Setting $S_2=U_{p_1}^{[p_1]}$ as before, we are reduced to the case $p_1=0$,
which is trivial.
 
In case 2, we observe that $U_{p_1}$ does not intersect $T_1$, whence there is a positive integer
$t\leq p_1$  such that $U^{[t]}_{p_1}$ intersects $T^{[t]}_1$, but $U^{[t-1]}_{p_1}$ does not intersect $T^{[t-1]}_1$.
We conclude that $l_3\in\{0,1\}$. Now the result follows by computing the diameter of $S_3$ as in Proposition
\ref{inter1}. Case 4 is similar.

\end{proof}

\begin{example}
If $\Ha$ is an order generated by an suborder $\Ha_0\cong\oink\times\oink$ and a nilpotent element $\epsilon$,
the branch $S=S_0(\Ha)$ is a path, and $l(S)\neq\infty$, since $S_0(\epsilon)$ does not contain a maximal path.
 It follows that there exist one or two maximal orders $\Da$,
containing $\Ha$, for which $\epsilon\notin\pi\Da$, the endpoints of $S$.
 There exists exactly one such order if and only if $l(S)=\infty/2$, i.e., $S_0(\Ha_0)$ contains
a long path of the infinite leaf $S_0(\epsilon)$. This holds if and only if $\Ha_0\big(\mathrm{Ker}(\epsilon)\big)=\mathrm{Ker}(\epsilon)$. Up to scalar multiples, there are exactly two nilpotent elements with this property.
\end{example}

\begin{proposition}\label{inter4}
Let $S_1$ and $S_2$ be two infinite leaves, and let $S_3=S_1\cap S_2$. Then $S_3$ is infinite if and only if
$S_1$ and $S_2$ have a comon long path. If this holds, then either $S_1\subseteq S_2$ or $S_2\subseteq S_1$, and in this case $S_3$ is an infinite leaf. If $S_3$ is finite, and the maximal depth $p(S_1,\Da)$ of a vertex $\Da$ of $S_3$ is $r$, then $p_3=[\frac r2]$, and $l_3=2\{\frac r2\}\in\{0,1\}$.
\end{proposition}

\begin{proof}
By inspecting the explicit list of possible branches of orders given in \S1, we conclude that every infinite branch contains an infinite path, and this can only be a long path in each of the branches $S_1$ and $S_2$. First statement follows. Assume now that $S_3$ is infinite and let $\Da$ be an arbitrary vertex of $S_3$. Assume $p(S_1,\Da)\geq p(S_2,\Da)$. Then if $U$ is the long path starting at $\Da$, and $\Da_i$ is the $i$-th vertex of that path we have
$$S_1=\bigcup_{i=0}^\infty B\Big(\Da_i;p\big(S_1,\Da\big)+i\Big)\supseteq \bigcup_{i=0}^\infty B\Big(\Da_i;p\big(S_2,\Da\big)+i\Big)=S_2.$$
Finally, assume $S_3$ is finite. Take a vertex $\Da\in S_3$ whose depth $r=p(S_1,\Da)$ is maximal. Since any path inside $S_3$ starting from $\Da$ goes outwards through the leaves in $S_1$, we have $S_3=S_2\cap B(\Da;r)$. Note also that $\Da$ is an endpoint of $S_2$, whence the result follows by setting $a=b=p_0=0$ and $p_1=r$ in case (2) of Proposition \ref{inter3}.
\end{proof}

\section{Relative position of the branches}

The following lemma, whose proof is straightforward,  is as close as we can get to give a general method to determine the relative position
of the branches:

\begin{lemma}\label{l31}
If the distance between two branches $S_1$ and $S_2$ is $d$, then $S_1^{[t_1]}$ and $S_2^{[t_2]}$ intersect
if and only if $t_1+t_2\geq d$. 
\end{lemma}

If we can write two orders $\Ha_1$ and $\Ha_2$ as the contractions $\Ea_1^{[s_1]}$ and $\Ea_2^{[s_2]}$
of two orders $\Ea_1$ and $\Ea_2$ whose branches are the stems of the branches
$S_1=S_0(\Ha_1)$ and $S_2=S_0(\Ha_2)$, the preceding lemma can be used to find the distance between the
corresponding stems, and thence the invariants for the spanned order $\Ha_3=\oink[\Ha_1,\Ha_2]$. 
If one of the branches is an infinite leaf, the same trick can be applied to find out the maximal depth of the intersection,
as in the following example:

\begin{example}
Let $\epsilon_1$ and $\epsilon_2$ be two nilpotent elements satisfying $\epsilon_1\epsilon_2+\epsilon_2\epsilon_1=\pi^2I$,
where $\pi$ is a uniformizing parameter of $k$ and $I$ is the identity matrix.
Then there is an order containing $\epsilon_1$ and $\frac{\epsilon_2}{\pi^t}$ if and only if $t\leq2$. It follows that
the greatest depth in $S_0(\epsilon_2)$ of an element in $S_0(\epsilon_1)$ is $2$. We conclude that $S=S_0(\epsilon_1,\epsilon_2)$ has the invariants $d(S)=2$, $l(S)=0$, and  $p(S)=1$.
\end{example}

It is 
sometimes better to replace one of the orders by a simpler order with the same branch (or an appropriate sub-branch),
in order to perform computations. The following example illustrate this point. 

\begin{example}
Assume $k=\mathbb{Q}_3$. Let $\Ha'=(\Da_1\cap\Da_2)^{[t]}$, where $\Da_1$ is the only 
maximal order containing $\eta=\lbmatrix 01{-1}0$ and $\Da_2$ is the only 
maximal order containing $\eta'=\lbmatrix 0{27}{-1/27}0$. Let $\epsilon=\lbmatrix 0{81}00$, so that
$\epsilon\eta+\eta\epsilon=-81I$ and $\epsilon\eta'+\eta'\epsilon=-3I$.
 Then, as before, $\Da_1$ has depth
1 in $S_0(\epsilon)$, while $\Da_2$ has depth 4. On the other hand, the relation 
$\eta\eta'+\eta'\eta=-\frac{730}{27}I$
shows that the distance between $S_0(\eta)=\{\Da_1\}$ and $S_0(\eta')=\{\Da_2\}$
 is exactly $3$. We conclude that the branch of $\Ha'$
has a stem of length 3 whose deepest point is the endpoint $\Da_2$ and it is at depth $4$. Therefore,
the invariants for the order $\Ha=\oink[\Ha',\epsilon]$ generated by $\Ha'$ and $\epsilon$ are
(c.f. Prop. \ref{inter3}):
\begin{itemize}
\item $p(\Ha)=t$, and $l(\Ha)=\mathrm{min}\{3,4-t\}$ if $t\leq4$,
\item $p(\Ha)=2+[\frac t2]$, and $l(\Ha)=2\{\frac t2\}$ otherwise.
\end{itemize}

\end{example}

Now we compute the distance between branches of standard generators. Before we do that, we prove a formula for
the branch of a pure quaternion. This is computed in terms of the quadratic defect of its square \cite{ohm}.
 Recall that the quadratic defect of
an element $a\in k$ is the fractional ideal $\mathfrak{d}(a)$ generated by all elements $\eta\in k$ satisfying $|a-u^2|_k\geq|\eta|_k$ for all $u\in k$.  There is always an element $u$ satisfying $(a-u^2)=\mathfrak{d}(a)$. Furthermore, for every uniformizing parameter $\pi$ we have $\mathfrak{d}(\pi)=(\pi)$, while
the unique  unramified quadratic extension of $k$ has the form $k(\sqrt\Delta)$ for any unit $\Delta$
satisfying $\mathfrak{d}(\Delta)=(4)$. The latter is called a unit of minimal quadratic defect.
 The quadratic defect of any other unit $u$ has the form $\mathfrak{d}(u)=(\pi)^{2t+1}$, where $0\leq t<e=v_k(2)$. 

\begin{lemma}
The depth $p=p(S)$ of the branch $S=S_0(q)$, for any pure quaternion $q$ such that $q^2=a\in \oink$, satisfies
$\mathfrak{d}(a)=(\pi)^{2p+1}$,  except in the following cases:
\begin{enumerate}
\item If $q$ generates an unramified quadratic extension, then  $\mathfrak{d}(a)=(\pi)^{2p}$.
\item If $a=b^2$, for some $b\in \oink$, then  $p=v_k(2b)$.
\end{enumerate}
\end{lemma}

\begin{proof}
Without loss of generality, we can assume $a$ is either a unit or a uniformizing parameter.
Assume $\mathfrak{d}(a)=(\pi)^{2t+1}$, so in particular $k(q)/k$ is a ramified extension. Then, there exists a unit
$v\in\oink^*$ and an integer $b\in\oink$ satisfying $a=b^2+v\pi^{2t+1}$. Let $\omega=\frac{q-b}{\pi^t}$. The quaternion $\omega$ is integral over $\oink$, since its norm is $\frac{b^2-q^2}{\pi^{2t}}=-\pi v$ and its trace is
 $\frac{-2b}{\pi^t}$, which is an integer since $t\leq v_k(2)$. 
 It follows  that $\oink[q]\subseteq \oink_{k(q)}^{[t]}$, and therefore $p\geq t$ by \cite[Prop. 2.4 and Prop. 4.2]{eichler2}.
 On the other hand, the branch of $\oink_{k(q)}$ is a path of length 1, and therefore its depth is $0$. 
If the branch of $\oink[q]$ has 
depth $p$, there must exists an element $b\in\oink$ such that $\frac{q-b}{\pi^p}$ is an integer \cite[Lemma 2.5]{eichler2}, and therefore
so is the norm $\frac{b^2-q^2}{\pi^{2p}}$. We conclude that $b^2-a\in (\pi^{2p})$, and therefore $p\leq t$.
The unramified case is similar. 

Assume now that $a=b^2$, so $k(q)$ can be identified with $k\times k$. Then we can assume $q=(b,-b)$, so that $\omega:=(0,1)=\frac1{2b}[(b,b)-q]$ is integral over $\oink$. Now the proof of the inequality $p\geq v_k(2b)$ goes as before. For the converse we observe that if $\frac{q-(b',b')}{\pi^p}$ is integral over $\oink$, then both $b-b'$ and
$-b-b'$ are in the ideal $(\pi^p)$, and therefore so is their difference $2b$.
\end{proof}

\begin{cor}
In the notations of the previous lemma, if $a=q^2$ is a unit, then
the depth satisfies $p\leq v_k(2)$, with equality if and only
if $k(q)$ is isomorphic to $k\times k$ or an unramified extension. If $a$ is a uniformizing parameter,
the invariants of $S_0(q)$ are $l=1$ and $p=0$.
\end{cor}

Next result give us the relative position of the branches for a cyclic order of the form $\oink[i,j]$ as in the introduction.

\begin{lemma}
Let $\Ha=\oink[i,j]$ be a cyclic order in a split quaternion algebra $\alge=K[i,j]$, for a pair of orthogonal pure quaternions
$i$ and $j$, such that $a=i^2$, and $b=j^2$, are both in $\oink^*\cup \pi\oink^*$.
\begin{enumerate}
\item Assume $a$ and $b$ are units,and let $s$ and $r$ be the depths of $S_0(i)$ and $S_0(j)$ respectively.
\begin{enumerate}
\item If $v_k(2)-s-r<0$, then the distance between the stems of $S_0(i)$ and $S_0(j)$ is $s+r-v_k(2)$.
\item If $v_k(2)-s-r=0$, then the stems of $S_0(i)$ and $S_0(j)$ intersect in a single point.
\item If  $v_k(2)-s-r>0$, the stems coincide.
\end{enumerate}
\item If $a$ is prime,  then the stem of $S_0(i)$ is contained in the stem of $S_0(j)$.
\end{enumerate}
\end{lemma}

\begin{proof}
Reasoning as in the proof of the previous lemma, there are elements $c,d\in\oink$ and integral elements $\eta\in k(i)$, $\omega\in k(j)$,
 such that $i=d+\pi^s\eta$ and $j=c+\pi^r\omega$. The relation $ij=-ji$ implies
\begin{equation}\label{etaom}
\omega\eta+\eta\omega=\frac2{\pi^{r+s}}cd+\frac2{\pi^r}c\eta+\frac2{\pi^s}d\omega.
\end{equation}
Assume first that $a$ and $b$ are units and the condition in (1a) is satisfied,
so that in particular $r$ and $s$ are  positive and therefore $c$ and $d$ are units. Then the first term on the 
right of equation (\ref{etaom}) is  dominant. In particular, if $\omega$ and $\pi^t\eta$ are in a maximal order,
the element $\pi^t(\omega\eta+\eta\omega)$ is integral, and therefore $t\geq  s+r-v_k(2)$. On the other hand,
if $t\geq  s+r-v_k(2)$, equation (\ref{etaom}) and the fact that $\eta$ and $\omega$ satisfy monic quadratic polynomials over $\oink$, proves that the lattice with basis $\{1,\pi^t\eta,\omega,\pi^t\eta\omega\}$ is a ring.
The result follows in Case (1a).

Assume now that $v_k(2)-s-r=0$. As before, there is a maximal order $\Da$ containing $\eta$ and $\omega$. We can
assume that $\Da=\matrici_2(\oink)$. Let $\bar{\omega}$ and $\bar{\eta}$ be the images of $\omega$ and $\eta$
in $\Da/\pi\Da\cong\matrici_2(\mathbb{K})$, where $\mathbb{K}$ is the residue field of $k$.
 To prove that $\Da$ is the unique order
containing $\eta$ and $\omega$, it suffices to show that $\bar{\omega}$ and $\bar{\eta}$ have no common eigenvector.
\begin{itemize}
\item Assume first that $k$ is dyadic, and $rs\neq0$. Then the equation for $\bar{\omega}$ and $\bar{\eta}$ becomes
$ \bar{\omega}\bar{\eta}+\bar{\eta}\bar{\omega}=\bar{u}\bar{c}\bar{d}$ where the bar denotes projection to the residue field, while $c$, $d$ and $u=\frac2{\pi^{r+s}}$, are all units. If $v\in\mathbb{K}^2$ is a comon eigenvector for $\bar{\omega}$ and $\bar{\eta}$ with corresponding eigenvalues $\lambda_{\omega}$ and $\lambda_{\eta}$, we have
 $0=\lambda_{\omega}\lambda_{\eta}+\lambda_{\eta}\lambda_{\omega}=\bar{u}\bar{c}\bar{d}$, a contradiction.
\item If $k$ is dyadic, but $rs=0$, we can assume $r=0$ and therefore we can also assume $j=\omega$, so $c=0$.
Note that $s\neq0$ and $u=2/\pi^s$ is a unit, since $r+s=v_k(2)$. The equation for $\bar{\omega}$ and $\bar{\eta}$ becomes $ \bar{\omega}\bar{\eta}+\bar{\eta}\bar{\omega}=\bar{u}\bar{d}\bar{\omega}$, 
and   the eigenvalue $\lambda_\omega$ 
is non-zero, as $\omega=j$ is the square root of a unit. Now the result follows as before.
\item If $k$ is non-dyadic, then $s=r=0$ and we can assume as before that  $j=\omega$ and $i=\eta$. The result
follows since in this case $\bar{\omega}$ and $\bar{\eta}$ are generators of a quaternion algebra, i.e., they generate
$\matrici_2(\mathbb{K})$.
\end{itemize} The result follows in Case (1b).

Assume now that $v_k(2)-s-r>0$, so that, in particular, $k$ is dyadic. This case is similar to the preceding one, except that the equation in the residual algebra is $ \bar{\omega}\bar{\eta}+\bar{\eta}\bar{\omega}=0$. The inequality $v_k(2)-s-r>0$ implies that both $i$ and $j$ generate ramified extensions and therefore the stems are paths of length 1. It suffices, as before, to prove that 
$\bar{\omega}$ and $\bar{\eta}$ have a common eigenvector, so $\eta$ and $\omega$ are contained simultaneously in more than one
maximal order. Note that $\bar{\omega}$ has an eigenvector,
since $\omega$ generates a ramified extension, and $\bar{\eta}$ leaves this eigenspace invariant, since it commutes
with $\bar{\omega}$. The result follows in Case (1c).

Finally we prove Case (2). In this case, the quadratic defect of $a$ is $(\pi)$, so we can set $i=\eta$, $s=0$, and $d=0$,
in equation (\ref{etaom}). In particular, $v_K(2)-s-r\geq 0$, so that equation (1) proves that $\eta$ and $\omega$ span
an order. It follows that  there exists a maximal order $\Da$ containing $\eta$ and $\omega$, and we can assume $\Da=\matrici_2(\oink_K)$. Furthermore, we can assume $i=\lbmatrix 0a10$,
and therefore $\bar{\eta}=\lbmatrix 0010$.
In this case the equation for $\bar{\omega}$ and $\bar{\eta}$ becomes
$ \bar{\omega}\bar{\eta}+\bar{\eta}\bar{\omega}=\bar{u}\bar{d}\bar{\omega}$, and by setting $\bar{\omega}=\lbmatrix xyzw$, the preceding equation gives $y=0$, so that $\lcvector01$ is a common eigenvector.
The result follows as before.
\end{proof}

\section{The tables}

\newcommand{\OOO}{\mathcal{O}}
\newcommand{\pe}{\mathfrak{p}}
\newcommand{\de}{\mathfrak{d}}

In this section we present complete tables for the invariants of an order of the form $\Ha=\OOO[\pi^ri,\pi^sj]$ where $i^2=\alpha$ and $j^2=\beta$ are units or primes, while $r$ and $s$ are non-negative integers. All computations are straightforward from Propositions 2.1-2.2 and Propositions 3.2-3.3.

 When $k$ is non-dyadic every unit is unramified.
In this case, the depth of the branches, $p_1=p(\OOO[i])$  and $p_2=p(\OOO[j])$ are zero. 
The invariants of $\Ha$ for all possible values of $\alpha$ and $\beta$ can be read from Table 1, 
by switching $\alpha$ and $\beta$ if needed.  Note that in the table, the case $\alpha\in\Delta\OOO^{*2}$,  
$\beta\in\pi\OOO^{*}$
is not shown. This is due to the
fact that $\left(\frac{\Delta,\pi}k\right)$ is always a division algebra. The intersection of the branches 
are shown in Figure 7.

\begin{table}[!hbp]
\small\[
\begin{tabular}[c]{|c|c|c|c|c|c|}\hline
$\alpha \in$ & $\beta \in $ & $r-s$ & $p(\h)$ & $l(\h)$ & $d(\h)$\\ \hline
{$\OOO^{*2}$} & $\OOO^{*2}$ &&  $\min\{r,s\}$ & $2|r-s|$ &$2\max\{r,s\}$\\ \hline
 {$\OOO^{*2}$}  & $ \Delta\OOO^{*2}$ &$\leq0$&$r$&  $2(s-r)$ &  $2s$  \\ \hline
 {$\OOO^{*2}$}  & $ \Delta\OOO^{*2}$ &$\geq0$&$s$&  $0$ &  $2s$  \\ \hline
 {$\OOO^{*2}$} & {$\pi\OOO^{*}$} &$\leq0$& $r$ & $2(s-r)+1$ & $2s+1$\\ \hline
 {$\OOO^{*2}$} & {$\pi\OOO^{*}$} &$\geq0$& $s$ & $1$ & $2s+1$\\ \hline
 {$\Delta\OOO^{*2}$} & $\Delta\OOO^{*2} $ & &$\min\{r,s\}$&$0$ & $2\min\{r,s\}$  \\ \hline
  {$\pi\OOO^{*}$} & $\pi\OOO^{*}$  && {$\min\{r,s\}$} & $1$ &  {$2\min\{r,s\} + 1$} \\ \hline
\end{tabular}
\]\normalsize
\caption{Invariants of the order $\Ha$ described in the text, when $k$ is non-dyadic.}
\end{table}

\begin{figure}
\[ \includegraphics[width=1\textwidth]{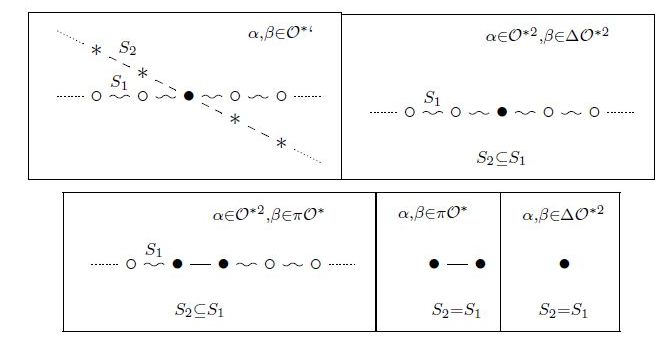}
\]
\caption{All possible configurations of the branches $S_1=S_0(i)$ and
$S_2=S_0(j)$ for a non-dyadic local field. Bullets and continuous lines $(\bullet\!-\!\!\!-\bullet)$
 denote the intersection $S_3=S_1\cap S_2$.}
\end{figure}

In order to simpliy Table 2, which contain the values for the invariants at
 general dyadic places, we define, when both $\alpha$ and $\beta$ are units, an auxiliar invariant $\epsilon$.  
Let  $$g=g(\alpha,\beta;r,s)=\delta(T_1,T_2)-\Big|p\big(S_0(\pi^ri)\big)-
p\big(S_0(\pi^sj)\big)\Big|,$$
where $T_1$ and $T_2$ are the stems of $S_0(\pi^ri)$ and $S_0(\pi^sj)$ respectively. 
Let
 $\mho=\OOO^*\backslash(\OOO^{*2}\cup\Delta\OOO^{*2})$ be the set of ramified units. When $\beta\in\mho$,
we define $t$ by $\mathfrak{d}(\beta)=(\pi)^{2t+1}$. Similarly, When $\alpha\in\mho$,
we set $\mathfrak{d}(\alpha)=(\pi)^{2u+1}$. In case $\delta(T_1,T_2)>0$, which is the only case where $\epsilon$
is needed, the number $g$ is given by
the formula (see Lemmas 3.2-3.3):
$$g=
\left\{\begin{array}{rcl}
t+u-e-|u+r-t-s|&\textnormal{ if }&\alpha,\beta\in\mho\\
t-|e+r-t-s|&\textnormal{ if }&\alpha\in \OOO^{*2}\cup\Delta\OOO^{*2},\beta\in\mho\\
e-|r-s|&\textnormal{ if }&\alpha,\beta\in \OOO^{*2}\cup\Delta\OOO^{*2}\end{array}\right..$$
In all cases we set $\epsilon=2\{g/2\}$, i.e., $\epsilon=0$ if $g$ is even and $1$ otherwise.
 Note that the invariant $t$ is defined as $0$ when $\beta\in\pi\OOO^{*}$. The same observation applies to $u$.

The last two lines in Table 2 deserve some additional explanation. The computation naturally breaks into three
cases acording to whether $t+u-e$ is positive, $0$ or negative (cf. Lemma 3.3). However, in each case the results
 coincide, except for the fact that the interval $[e-2u,2t-e]$ is $\{0\}$ when $t+u-e=0$, and empty when
$t+u-e<0$.

\begin{table}[!hbp]
\scriptsize\[
\begin{tabular}[c]{|c|c|c|c|c|c|}\hline
 $\alpha \in $ & $\beta \in $ &$(r-s)\in$ & $p(\h)$ & $l(\h)$ & $d(\h)$\\ \hline
 $\OOO^{*2}$ & $\OOO^{*2}$ &$[-e,e]^c$
&  $\min\{r,s\}+e$ &$2w$&$2\max\{r,s\}$\\ \hline
 $\OOO^{*2}$ & $\OOO^{*2}$ &$[-e,e]$& $\frac12(e+r+s-\epsilon)$ 
 & $\epsilon$ &  $e+r+s$  \\ \hline
 $\OOO^{*2}$ & $ \Delta\OOO^{*2}$ &$[-\infty,-e]$&$e+r$ & $2w$  & $2s$ \\ \hline
 $\OOO^{*2}$ &  $ \Delta\OOO^{*2}$ &$[-e,e]$&$\frac12(e+r+s-\epsilon)$
 & $\epsilon$  & $e+r+s$ \\ \hline
 $\OOO^{*2}$ &  $ \Delta\OOO^{*2}$ &$[e,\infty]$&$e+s$ & $0$  & $2(e+s)$ \\ \hline
 $\OOO^{*2}$ & {$\pi\OOO^{*}$} &$[-\infty,-e]$& $e+r$ & $2w+1$ &$2s+1$\\ \hline
 $\OOO^{*2}$ &{$\pi\OOO^{*}$} &$[-e,\infty]$& $s$ & $1$ &$2s+1$\\ \hline
 $\OOO^{*2}$ & {$\mho$} &$[-\infty,-e]$&$e+r$ & $2w$  & $2s$ \\ \hline
 $\OOO^{*2}$ & {$\mho$} &$[-e,-e+2t]$& $\frac12(e+r+s-\epsilon)$ & $\epsilon$ &$e+r+s$\\ \hline
 $\OOO^{*2}$ & {$\mho$} &$]-e+2t,\infty]$&$t+s$ & $1$  & $2(t+s)+1$ \\ \hline
 $ \Delta\OOO^{*2}$ & $ \Delta\OOO^{*2}$ &$[-e,e]^c$&  
$\min\{r,s\}+e$ &$0$&$2\min\{r,s\}+2e$\\ \hline
 $ \Delta\OOO^{*2}$ & $ \Delta\OOO^{*2}$ &$[-e,e]$
& $\frac12(e+r+s-\epsilon)$  & $\epsilon$ &  $e+r+s$  \\ \hline
 $ \Delta\OOO^{*2}$ & {$\mho$} &$[-\infty,-e]$&$e+r$ & $0$  & $2(e+r)$ \\ \hline
 $ \Delta\OOO^{*2}$ & {$\mho$} &$[-e,-e+2t]$& $\frac12(e+r+s-\epsilon)$ 
& $\epsilon$ &$e+r+s$\\ \hline
 $ \Delta\OOO^{*2}$ & {$\mho$} &$]-e+2t,\infty]$&$t+s$ & $1$  & $2(t+s)+1$ \\ \hline
 $\pi\OOO^{*}$ & $\pi\OOO^{*}$ &$\mathbb{R}$&$\min\{r,s\}$ & $1$  & $2\min\{r,s\}+1$ \\ \hline
 $\pi\OOO^{*}$ & $\mho$ &$\mathbb{R}$&$\min\{r,t+s\}$ & $1$  & $2\min\{r,t+s\}+1$ \\ \hline
  {$\mho$} & {$\mho$} &$[e-2u,2t-e]^c$&  $\min\{r+u,s+t\}$ &$1$&$2\min\{r+u,s+t\}+1$\\ \hline
  {$\mho$} & {$\mho$} &$[e-2u,2t-e]$ &  $\frac12(e+r+s-\epsilon)$ &$\epsilon$&$e+r+s$\\ \hline
\end{tabular}
\]\normalsize
\caption{Invariants of $\Ha$ when $k$ is dyadic. Here $w=|s-r|-e$.
 The numbers $t$, $u$, and  $\epsilon$ are defined in the text. $I^c$ is the complement of the interval $I$.
 The interval in the last line can be empty.}
\end{table}

\begin{table}[!hbp]
\scriptsize\[
\begin{tabular}[c]{|c|c|c|c|c|c|}\hline
 $\alpha \in $ & $\beta \in $ &$(r-s)\in$ & $p(\h)$ & $l(\h)$ & $d(\h)$\\ \hline
 $\OOO^{*2}$ & $\OOO^{*2}$ &$\mathbb{R}-\{0\}$
&  $\min\{r,s\}+1$ &$2|r-s|-2$&$2\max\{r,s\}$\\ \hline
 $\OOO^{*2}$ & $\OOO^{*2}$ &$\{0\}$& $r$ 
 & $1$ &  $2r+1$  \\ \hline
 $\OOO^{*2}$ & $ \Delta\OOO^{*2}$ &$[-\infty,-1]$&$r+1$ & $2(s-r-1)$  & $2s$ \\ \hline
 $\OOO^{*2}$ &  $ \Delta\OOO^{*2}$ &$\{0\}$&$r$
 & $1$  & $2r+1$ \\ \hline
 $\OOO^{*2}$ &  $ \Delta\OOO^{*2}$ &$[1,\infty]$&$s+1$ & $0$  & $2(s+1)$ \\ \hline
 $\OOO^{*2}$ & {$\pi\OOO^{*}$} &$[-\infty,1]$& $s$ & $1$ &$2s+1$\\ \hline
 $\OOO^{*2}$ &{$\pi\OOO^{*}$} &$[1,\infty]$& $r+1$ & $2(s-r)-1$ &$2s+1$\\ \hline
 $\OOO^{*2}$ & {$\mho$} &$[-\infty,-1]$&$r+1$ & $2(s-r-1)$  & $2s$ \\ \hline
 $\OOO^{*2}$ & {$\mho$} &$[0,\infty]$&$s$ & $1$  & $2s+1$ \\ \hline
 $ \Delta\OOO^{*2}$ & $ \Delta\OOO^{*2}$ &$\mathbb{R}-\{0\}$&  
$\min\{r,s\}+1$ &$0$&$2\min\{r,s\}+2$\\ \hline
 $ \Delta\OOO^{*2}$ & $ \Delta\OOO^{*2}$ &$\{0\}$
& $r$  & $1$ &  $2r+1$  \\ \hline
 $ \Delta\OOO^{*2}$ & {$\mho$} &$[-\infty,-1]$&$r+1$ & $0$  & $2(r+1)$ \\ \hline
 $ \Delta\OOO^{*2}$ & {$\mho$} &$[0,\infty]$&$s$ & $1$  & $2s+1$ \\ \hline
 $\pi\OOO^{*}$ & $\pi\OOO^{*}\cup\mho$ &$\mathbb{R}$&$\min\{r,s\}$ & $1$  & $2\min\{r,s\}+1$ \\ \hline
$\mho$ & $\mho$ &$\mathbb{R}$&$\min\{r,s\}$ & $1$  & $2\min\{r,s\}+1$ \\ \hline
\end{tabular}
\]\normalsize
\caption{Invariants of $\Ha$ when $k$ is a dyadic field that is unramified over $\mathbb{Q}_2$.}
\end{table}

\begin{figure}
\[ \includegraphics[width=1\textwidth]{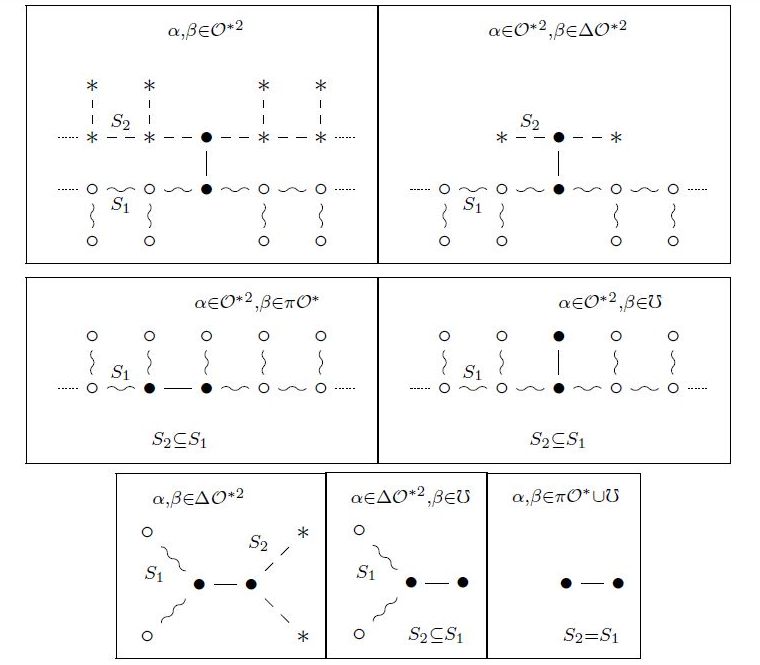}
\]
\caption{All possible configurations of the branches $S_1=S_0(i)$ and
$S_2=S_0(j)$ for an un-ramified dyadic local field. Bullets and continuous lines denote the 
intersection $S_3=S_1\cap S_2$. In the pictures we assume $q=|\oink/\pi\oink|=2$.}
\end{figure}

There is a significant number of possible configurations when $k$ is non-dyadic, and we desist
from the task of drawing them all, but we actually do it in one important case, when $k$ is
an unramified extension of $\mathbb{Q}_2$ (Table 3 and Fig. 8).
In the latter case, $e=1$, so every ramified unit has quadratic defect
$(\pi)$.
A quick glance to the case $r=s=0$ in Table 3, or Figure 8  shows the following important result:
\begin{proposition}\label{p44}
When  $k$ is a dyadic field that is unramified over $\mathbb{Q}_2$, every order of the form
$\OOO[i,j]\subseteq\mathbb{M}_2(k)$, where $ij=-ji$ and $i^2,j^2\in \oink$ are square free, is contained
in exactly $2$ maximal orders.
\end{proposition}

\section{Examples and applications}

\begin{example}\label{ex1}
If $k=\mathbb{Q}_2(\sqrt{-1})$, then $\pi=1+\sqrt{-1}$ is a uniformizing parameter. The quadratic defect
of $u=1+2\pi=3+2\sqrt{-1}$ is $(\pi)^3$, since $1$ and $3$ are the only squares of units in $\oink/4\oink$.
 In Figure 11 we see the branch $S_0(i,j)$ for
$j^2=u$, and different values of $\alpha=i^2$. As usual $S_1=S_0\big(i)$ and
$S_2=S_0\big(j)$. Recall that, as $-1$ is a square in $k$,
 the quaternion algebra $\left(\frac{1+\pi,1+\pi}{k}\right)\cong
\left(\frac{1+\pi,\pi}{k}\right)$ splits.

\begin{figure}\[
\includegraphics[width=1\textwidth]{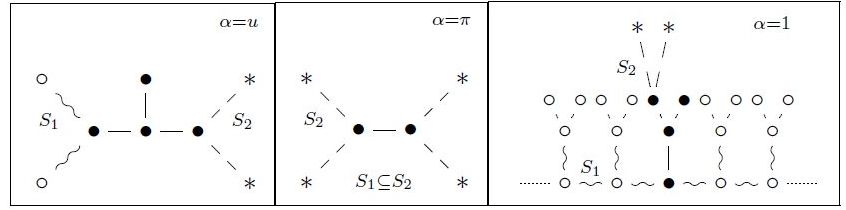}
\]
\caption{Some configurations of the branches $S_1=S_0(i)$ and
$S_2=S_0(j)$ in example \ref{ex1}.}
\end{figure}
\end{example}


\begin{example}
The tables can be used to compute the branch of an order of the form $\oink[\eta,\omega]$ for any pair of integral
elements satisfying $k[\eta]=k[i]$ and $k[\omega]=k[j]$, as long as they generate an order. 
 For example, if $k=\mathbb{Q}_2(\sqrt{-1})$,
 to find the invariants of the order generated by $\frac{i+1}{\pi}$ and $\frac{j+1}{\pi}$ when
 $i^2=j^2=u$, as in Example 4, we just set $r=s=-1$ in Table 2.
\end{example}

\begin{example}\label{ex2}
The quaternion algebra $\alge=\left(\frac{-3,-3}{\mathbb{Q}}\right)$ ramifies in exactly $2$ places, $3$ and 
$\infty$. The order $\Ha=\mathbb{Z}[i,j]\subseteq\alge$, where $ij=-ji$ and $i^2=j^2=-3$ 
is maximal outside the set $\{2,3\}$.
The local order $\mathbb{Z}_2[i,j]$ is contained in exactly $2$ maximal orders. Since the maximal order in $\alge_3$
is unique, there are exactly $2$ global maximal orders, $\Da$ and $\Da'$, containing $\Ha$. In fact,
 $\Da=\enteri[i,\omega,\nu]$ and $\Da'=\enteri[\eta,j,\nu]$, where $j=2\omega+1$, $i=2\eta+1$, and $3\nu=ij$. 
It can be proved easily, using the fact that these orders are neighbors at $2$, that all maximal
orders in $\alge$ are conjugate. See \cite{nocubic} for details.
\end{example}

\subsection{Spinor image Computations and representation fields.}
Here we recall the basic facts in the theory of representation fields, see \cite{abelianos} or \cite{cyclic}  for details.
The set of spinor genera in a genus $\mathbb{O}=\mathrm{gen}(\Da)$ 
of Eichler orders in a quaternion algebra $\alge$
 over a global field $K$ equals $[\Sigma:K]$, where $\Sigma=\Sigma(\mathbb{O})$, the spinor class field, is the
class field corresponding to the class group $K^*H(\Da)$, where
$$H(\Da)=\{N(a)|a\in\alge_\ad,\ a\Da_\ad a^{-1}=\Da_\ad\},$$
if $N$ denotes the adelic reduced norm, while $\alge_\ad$ and $\Da_\ad$ are the adelizations 
of the algebra $\alge$ and the order $\Da$. The field $\Sigma$ can also be described as the largest abelian
extension of $K$ ramifying only at real places that are ramified for $\alge$, and splitting at every 
finite place $\wp$ satisfying either of the following conditions:
\begin{itemize}
 \item $\alge$ is ramified at $\wp$, or
\item the level of $\Da$ at $\wp$ is odd.
\end{itemize} 
 The number of  spinor genera representing a given order
$\Ha$ equals $[\Sigma:F]$, where $F$ is the representation field, i.e.,  the
class field corresponding to the group $K^*H(\Da|\Ha)$, where
$$H(\Da|\Ha)=\{N(a)|a\in\alge_\ad,\ a\Ha_\ad a^{-1}\subseteq \Da_\ad\}.$$
For every finite place $\wp$ it is easy to see that the local component $H_\wp(\Da|\Ha)$, which is defined analogously,
is either $K_\wp^*$ or $\oink_\wp^*K_\wp^{*2}$. In fact $H_\wp(\Da|\Ha)=\oink_\wp^*K_\wp^{*2}$
 if and only if the level of $\Da$ at $\wp$ is even and equals the local diameter $d(\Ha_\wp)$.

An order $\Ha$ embeds into every spinor genera of maximal orders if and only if $F=K$. When this is not the case,
we say that $\Ha$ is selective \cite{FriedmannQ}, \cite{lino1}.

\begin{example}\label{ex2}
Let $K=\mathbb{Q}(\sqrt{-15})$, whose Hilbert class field is $\Sigma=\mathbb{Q}(\sqrt{-3},\sqrt5)$ 
\cite[p. 262]{cohn}. In particular, there are exactly $2$ conjugacy classes of maximal orders in $\mathbb{M}_2(K)$.
Note that $3$ and $5$ ramify on $K/\mathbb{Q}$. Let  $\mathbf{3}_0$ and $\mathbf{5}_0$ denote the corresponding places of $K$, which are inert on $\Sigma/K$. On the other hand, $2$ splits on $K/\mathbb{Q}$. The two dyadic places
$\mathbf{2}_1$ and $\mathbf{2}_2$ are inert on $\Sigma/K$. Set $\Ha=\OOO[i,j]$, where $ij=-ji$, while
$i^2=-3$ and $j^2=-15$. Then, according to the tables, the invariants of the local branches of $\Ha$ are as follows:
$$\Big(d(\Ha_{\wp}),p(\Ha_{\wp}),l(\Ha_{\wp})\Big)=\left\{\begin{array}{rcl}
(2,1,0)&\textnormal{ if }&\wp=\mathbf{3}_0\\
(2,0,2)&\textnormal{ if }&\wp=\mathbf{5}_0\\
(1,0,1)&\textnormal{ if }&\wp\in\{\mathbf{2}_1,\mathbf{2}_2\} \\
(0,0,0)&&\textnormal{otherwise}\end{array}\right..$$
It follows that $\Ha$ embeds into an Eichler order of every level dividing $(30)$. Furthermore, since
the Frobenius element at each of the relevant places
$$|[\mathbf{3}_0,\Sigma/K]|=|[\mathbf{5}_0,\Sigma/K]|=|[\mathbf{2}_1,\Sigma/K]|=|[\mathbf{2}_2,\Sigma/K]|$$
is the nontrivial element in $\mathrm{Gal}(\Sigma/K)$, there is a unique conjugacy class of Eichler orders of level $(30)$,
but $2$ conjugacy classes of Eichler orders of level $(15)$ and $\Ha$ embeds in both of them. On the other hand, for the order
$\Ha'=\OOO[i,2j]$, the local invariants are:
$$\Big(d(\Ha'_{\wp}),p(\Ha'_{\wp}),l(\Ha'_{\wp})\Big)=\left\{\begin{array}{rcl}
(2,1,0)&\textnormal{ if }&\wp\in\{\mathbf{3}_0\mathbf{2}_1,\mathbf{2}_2\}\\
(2,0,2)&\textnormal{ if }&\wp=\mathbf{5}_0\\
(0,0,0)&&\textnormal{otherwise}\end{array}\right.,$$
whence $\Ha'$ is selective on the genus of Eichler orders of level $(60)$.
\end{example}

\subsection{Number of maximal orders containing a given order.} It is not hard to show by a simple inductive argument
that the number of maximal orders in a thick line $S$ with invariants $p(S)$ and $l(S)$ is 
$$\aleph(S)=[l(S)+1]q^{p(s)}+2\frac{q^{p(s)}-1}{q-1},$$ where $q$ is the number of elements in the residual field.
It follows that the number of maximal orders containing an order $\Ha$ of maximal rank, in a quaternion algebra over a global field $F$, is
$\prod_{\wp}\aleph\big(S_0(\Ha_\wp)\big)$, where the product is taken over all finite places of $F$.
 Note that almost every factor in this product is one.
An important case, which is straightforward from Proposition \ref{p44} is next result:
\begin{proposition}
If $F/\mathbb{Q}$ is unramified at $2$, and $a,b\in\mathcal{O}_F$ are square free, the number of maximal orders in
an $F$-algebra  $\alge\cong\left(\frac{a,b}{F}\right)$ containing the order $\Ha=\mathcal{O}_F[i,j]$,
 where $ij=-ji$, $i^2=a$, and $j^2=b$, is $2^T$, where $T$
is the number of un-ramified places dividing $2ab$.
\end{proposition}

\subsection{Set of fixed points for groups of Moebius transformations.}
For any non-archimedean local field $k$, the group of Moebius transformations $\mathrm{PGL}_2(k)$
 acts on the vertices of
the BT-tree by conjugation. The stabilizer of a maximal order $\Da$ is the group $k^*\Da^*$. Note that an element
$u$ is a unit in some maximal order if and only if its norm is a unit and its trace is an integer,  and in this case it is a unit in every order containing it. It follows that the class $\bar u\in  \mathrm{PGL}_2(k)$ of a unit $u\in\mathrm{GL}_2(k)=\matrici_2(k)^*$ stabilizes every element in $S_0(u)$. On the other hand, if
$\bar u\in  \mathrm{PGL}_2(k)$ is not the class of a unit, then $\bar u$ cannot leave any vertex invariant. 
 In particular, if a subgroup of $\mathrm{PGL}_2(k)$ has a non-empty invariant locus in the BT-tree, then this locus is a set in the family that we described in \S1. The results in this work allow us to compute this invariant locus for isomorphic copies of the Klein group inside $\mathrm{PGL}_2(k)$.

\subsection{Generating sets of orders.}
 It is apparent from Propositions 2.1-2.4 that whenever the intersection of two branches has a stem of length 2 or higher,
this stem is contained in the stem of one of the intersecting branches. As a consequence, we conclude that, if $\Ha=\OOO[a_1,\dots,a_n]$ is an order whose branch has a stem of length 2 or higher, then at least one of the generators $a_i$ spans
an algebra isomorphic to $k\times k$. The converse is, however, not true since the conclusion holds for an Eichler order
$\Da$ of level $1$. In fact, $\Da/\pi\Da\cong(\mathbb{F}_k\times\mathbb{F}_k)\oplus R$, where $R$ is a radical and $\mathbb{F}_k$ is the residue field of $k$. Note that any generating set of this algebra must contain an element whose projection to $\mathbb{F}_k\times\mathbb{F}_k$ is a generator of the latter algebra.

\section{Acknowledgements}

The first author was supported by Fondecyt, Grant number 1140533. The second author was partly supported by
Fondecyt,  Grant number 1100127.

\end{document}